\documentclass[12pt]{amsart} %a4paper

\usepackage{amssymb}
\usepackage{amsmath}
\usepackage{xypic}

\theoremstyle{plain}
\newtheorem {lemma}{Lemma}[]
\newtheorem {theorem}[lemma]{Theorem}

\newtheorem {corollary}[lemma]{Corollary}

\newtheorem {proposition}[lemma]{Proposition}

\theoremstyle{remark}
\newtheorem* {remark}{Remark}

\newtheorem {remarks}[lemma]{Remarks}

\theoremstyle{definition}

\newcommand\CK[1][1]{\operatorname{CK}_{#1}}
\newcommand\SK[1][1]{\operatorname{SK}_{#1}} % This is $D^{(1)}/D'$.

\DeclareMathOperator{\End}{End}

\newcommand{\chr}{\operatorname{char}}

\newcommand{\Nrd}[1][{}]{{\operatorname{Nrd}_{#1}}} % The reduced norm
\newcommand{\Trd}[1][{}]{{\operatorname{Trd}_{#1}}} % The reduced trace
\newcommand\tensor[1][{}]{{\otimes_{#1}}}

\newcommand\GL[1][d]{{\operatorname{GL}_{#1}}} % General linear group
\newcommand\SL[1][d]{{\operatorname{SL}_{#1}}}

\newcommand\paper[1]{{\it {#1}}}

\long\def\forget#1\forgotten{}
\begin{document}

\title[$\SK$ of Azumaya algebras]{$\SK$ of Azumaya algebras over Hensel Pairs}

\author{Roozbeh Hazrat}
\address{
Dept. of Pure Mathematics\\
Queen's University\\
Belfast BT7 1NN\\
United Kingdom} \email{r.hazrat@qub.ac.uk}

\date{}

\begin{abstract}
Let $A$ be an Azumaya algebra of constant rank $n^2$ over a Hensel
pair $(R,I)$ where $R$ is a semilocal ring with $n$ invertible in
$R$. Then the reduced Whitehead group $\SK(A)$ coincides with its
reduction $\SK(A/IA)$. This generalizes a result of
\cite{azumayask1} to non-local Henselian rings.

\end{abstract}

\maketitle

Let $A$ be an Azumaya algebra over a ring $R$ of constant rank
$n^2$. Then there is an \'etale faithfully flat commutative ring $S$ over $R$ which splits $A$, i.e., $A\otimes_R S \cong M_n(S)$. For $a \in A$, considering $a \otimes 1$ as an element of $M_n(S)$, one then defines the reduced characteristic polynomial of $a$ as $$\chr_A(x,a)=\det(x-a\otimes1)=x^n-\Trd(a)x^{n-1}+\cdots+(-1)^n\Nrd(a).$$
Using descent theory, one can show that $\chr_A(x,a)$ is independent of $S$ and the isomorphism above and lies in $R[x]$. Furthermore, the element $a$ is invertible in $A$ if and only if $\Nrd_A(a)$, the reduced norm of $a$,  is invertible in $R$ (see~\cite{knus}, III.1.2, and ~\cite{saltman}, Theorem ~4.3). Let $\SL[](1,A)$ be the set of elements of $A$ with the reduced norm $1$. Since the reduced norm map respects the scalar extensions, it defines the smooth group scheme $\SL[]_{1,A}:T\rightarrow \SL[](1,A_T)$ where $A_T=A\otimes_R T$ for an $R$-algebra $T$. Consider the short exact sequence of smooth group schemes 
$$1 \longrightarrow \SL[]_{1,A} \longrightarrow \GL[]_{1,A} \stackrel{\Nrd}{\longrightarrow} G_m \longrightarrow 1$$ where $\GL[]_{1,A}:T\rightarrow A_T^*$ and 
$G_m(T)=T^*$ for an  $R$-algebra $T$ where $A_T^*$ and $T^*$ are invertible elements of $A_T$ and $T$, respectively.  This exact sequence induces a long exact sequence 
\begin{equation}\label{longexactse}
1 \longrightarrow \SL[](1,A) \longrightarrow A^* \stackrel{\Nrd}{\longrightarrow} R^* \longrightarrow H^1_{et}(R,\SL[](1,A)) \longrightarrow H^1_{et}(R,\GL[](1,A)) \rightarrow \cdots 
\end{equation}
Let $A'$
denote the commutator subgroup of $A^*$. One defines the reduced
Whitehead group of $A$ as $\SK(A)=\SL[](1,A)/A'$ which is a
subgroup of (non-stable) $K_1(A)=A^*/A'$.
Let $I$ be an ideal of $R$. Since the reduced norm is compatible with extensions, it induces the map $\SK(A) \rightarrow \SK(\overline A)$, where $\overline A=A/IA$.  A natural question arises here is, under what circumstances and
for what ideals $I$ of $R$, this homomorphism would be injective and/or surjective and thus the reduced Whitehead group of $A$
coincides with its reduction. The following observation shows that
even in the case of a split Azumaya algebra, these two groups
could differ: consider the split Azumaya algebra $A=M_n(R)$ where
$R$ is an arbitrary commutative ring (and $n>2$). In this case the reduced
norm coincides with the ordinary determinant and $\SK(A)=
\text{SL}_n(R)/[\text{GL}_n(R),\text{GL}_n(R)]$. There are
examples such that $\SK(A) \neq 1$, in fact not even torsion. But
in this setting, obviously $\SK(\overline A) =1$ for $\overline A=A/mA$
where $m$ is a maximal ideal of $R$ (for some examples see \cite{rosen}, Chapter 2).

If $I$ is contained in the Jacobson radical $J(R)$, then $I A
\subset J(A)$ (see, e.g., \cite{greco2}, Lemma 1.4) and
(non-stable) $K_1(A)\rightarrow K_1(\overline A)$ is surjective, thus its restriction to $\SK$ is also surjective. 

It is observed by Grothendieck  (\cite{groth1}, Theorem 11.7) that if $R$ is a local Henselian ring with maximal ideal $I$ and $G$ is an affine, smooth group scheme, then $H^1_{et}(R,G)\rightarrow H^1_{et}(R/I,G/IG)$ is an isomorphism. This was further extended to Hensel pairs by Strano~\cite{strano1983}. Now if further $R$ is a semilocal ring then $H^1_{et}(R,\GL[](1,A))=0$, and thus from the sequence (\ref{longexactse}) we have the following commutative diagram:
\begin{equation}\label{diag}
\begin{split}
\xymatrix{
& & (1+IA)A'/A'  \ar[r] \ar[d] & 1+I \ar[d]& \\
1 \ar[r] & \SK(A) \ar[r] \ar[d] & K_1(A) \ar[r]^{\Nrd}
\ar[d]& R^*
\ar[r] \ar[d] & H^1_{et}(R,\SL[](1,A)) \ar[r] \ar[d]^{\cong} & 1 \\
1  \ar[r] & \SK(\overline A) \ar[r]  & K_1(\overline A) \ar[r]^{\Nrd} \ar[d] &  {\overline R}^* \ar[r]
\ar[d] & H^1_{et}(\overline R,\SL[](1,\overline A)) \ar[r] & 1\\
& & 1 & 1 &}
 \end{split}
\end{equation}
The aim of this note is to prove that for the Hensel pair $(R,I)$ where $R$ is a semilocal ring, the map $\SK(A)\rightarrow \SK(\overline A)$ is also an isomorphism. This extends a result of \cite{azumayask1} to non-local Henselian rings.

%In \cite{azumayask1}, it was proved that for an Azumaya algebra
%$A$ over a Henselian local ring $R$ with the maximal ideal $m$,
%$\SK(A) \cong \SK(\overline A)$, where $\overline A= A/mA$, and $\mychar
%(\overline R)$ does not divide $n$, the index of $A$. In this note we
%extend the result to a Hensel pair $(R,I)$ where $R$ is a
%semilocal ring. 
Recall that the pair $(R,I)$ where $R$ is a
commutative ring and $I$ an ideal of $R$ is called a Hensel pair
if for any polynomial $f(x) \in R[x]$,  and $b \in R/I$ such that
$\overline f(b)=0$ and $\overline f'(b)$ is invertible in $R/I$, then there
is $a \in R$ such that $\overline a=b$ and $f(a)=0$ (for other
equivalent conditions, see Raynaud~\cite{raynaud}, Chap. XI).

In order to prove the statement, we use a result of Vaserstein~\cite{Vas} which
establishes the (Dieudonn\`e) determinant in the setting of
semilocal rings. The crucial part is to prove a version of
Platonov's congruence theorem \cite{plat} in the setting of an
Azumaya algebra over a Hensel pair. The approach to do this was
motivated by Suslin in \cite{suslin}. We also need to use the
following facts established by Greco in \cite{greco1,greco2}.

\begin{proposition}[\cite{greco2}, Prop. 1.6] \label{prop1}
Let $R$ be a commutative ring, $A$ be an $R$-algebra, integral
over $R$ and finite over its center. Let $B$ be a commutative $R$-subalgebra
of $A$ and $I$ an ideal of $R$. Then $IA \cap B \subseteq
\sqrt{IB}$.
\end{proposition}

\begin{corollary}[\cite{greco1}, Cor. 4.2] \label{cor1}

Let $(R,I)$ be a Hensel pair and let $J \subseteq \sqrt{I}$ be an
ideal of $R$. Then $(R,J)$ is a Hensel pair.
\end{corollary}

\begin{theorem}[\cite{greco1}, Th. 4.6] \label{thm1}
Let $(R,I)$ be a Hensel pair and let $B$ be a commutative
$R$-algebra integral over $R$. Then $(B,IB)$ is a Hensel pair.
\end{theorem}

We are in a position to prove the main theorem of this note.
\begin{theorem}\label{main}
Let $A$ be an Azumaya algebra of constant rank $n^2$ over a Hensel
pair $(R,I)$ where $R$ is a semilocal ring with $n$ invertible in
$R$. Then $\SK(A) \cong \SK(\overline A)$ where $\overline A=A/IA$.
\end{theorem}
\begin{proof}

Since for any $a \in A$, $\overline{\Nrd_A(a)}=\Nrd_{\overline A}(\overline
a)$, it follows that there is a homomorphism $\phi:\SL[](1,A)
\rightarrow \SL[](1,\overline A)$. We first show that $\ker \phi
\subseteq A'$, the commutator subgroup of $A^*$. In the setting of
valued division algebras, this is the Platonov congruence theorem
\cite{plat}. We shall prove this in several steps. Clearly $\ker
\phi= \SL[](1,A) \cap 1+IA$. Note that $A$ is a free $R$-module (see~\cite{bourbaki}, II, \S5.3, Prop. 5) . 

{\bf (i)}{\it . The group $1+I$ is uniquely $n$-divisible and $1+IA$ is
$n$-divisible.}

Let $a \in 1+I$. Consider $f(x)=x^n-a \in R[x]$. Since $n$ is
invertible in $R$, $\overline f(x)=x^n-1 \in \overline R[x] $ has a simple
root. Now this root lifts to a root of $f(x)$ as $(R,I)$ is a
Hensel pair. This shows that $1+I$ is $n$-divisible. Now if
$(1+a)^n=1$ where $a \in I$, then
$a(a^{n-1}+na^{n-2}+\cdots+n)=0$. Since the second factor is
invertible, $a=0$, and it follows that $1+I$ is uniquely
$n$-divisible.

 Now let $a \in 1+IA$. Consider the commutative ring $B=R[a]
\subseteq A$. By Theorem~\ref{thm1}, $(B,IB)$ is a Hensel pair. On
the other hand by Prop.~\ref{prop1}, $IA \cap B \subseteq
\sqrt{IB}$. Thus by Cor.~\ref{cor1}, $(B,IA \cap B)$ is also a
Hensel pair. But $a \in 1+IA \cap B$. Applying the Hensel lemma as
in the above, it follows that $a$ has a $n$-th root and thus
$1+IA$ is $n$-divisible.

{\bf (ii).} {\it $\Nrd_A(1+IA)=1+I$.}

From compatibility of the reduced norm, it follows that
$\Nrd_A(1+IA) \subseteq 1+I$. Now using the fact that $1+I$ is
$n$-divisible, the equality follows.

{\bf (iii).} {\it $\SK(A)$ is $n^2$-torsion}.

We first establish that $N_{A/R}(a)=\Nrd_A(a)^n$. One way to see
this is as follows. Since $A$ is an Azumaya algebra of constant
rank $n^2$,  $i:A \otimes A^{op} \cong \End_R(A)\cong
M_{n^2}(R)$ and there is an \'etale faithfully flat $S$ algebra
such that $j:A\otimes S \cong M_n(S)$.
  Consider the following diagram
$$\xymatrix{
A\otimes A^{op}\otimes S \ar[r]^-{i\otimes 1 } \ar[d]& \End_R(A)
\otimes
S \ar[r]^-{\cong} & \End_S(A\otimes S) \ar[r]^-{\cong} & M_{n^2}(S) \ar@{.>}[d]^{\psi}\\
A^{op}\otimes A \otimes S \ar[r]^-{1\otimes j } & A^{op} \otimes
M_n(S) \ar[r]^-{\cong}& M_{n}(A^{op} \otimes S) \ar[r]^-{\cong} &
M_{n^2}(S) }$$ where the automorphism $\psi$ is the compositions
of isomorphisms in the diagram. By a theorem of Artin (see, e.g.,
\cite{knus}, \S III, Lemma 1.2.1), one can find an  \`etale
faithfully flat $S$ algebra $T$ such that $\psi \otimes 1:
M_{n^2}(T) \rightarrow M_{n^2}(T)$ is an inner automorphism. Now
the determinant of the element $a\otimes 1 \otimes 1$ in the first
row is $N_{A/R}(a)$ and in the second row is $\Nrd_{A}(a)^n$ and
since $\psi \otimes 1$ is inner, thus they coincide.

Therefore if $a \in \SL[](1,A)$, then $N_{A/R}(a)=1$. We will show
that $a^{n^2} \in A'$. Consider the sequence of $R$-algebra
homomorphism $$ f:A \rightarrow A \otimes A^{op} \rightarrow
\End_R(A)\cong M_{n^2}(R) \hookrightarrow M_{n^2}(A)$$ and the
$R$-algebra homomorphism $i: A \rightarrow M_{n^2}(A)$ where $a$
maps to $a I_{n^2}$, where $I_{n^2}$ is the identity matrix of
$M_{n^2}(A)$. Since $R$ is a semilocal ring, the Skolem-Noether
theorem is present in this setting (see~\cite{knus}, Prop. 5.2.3) and thus there is $g \in \GL[n^2](A)$ such that
$f(a)=g i(a) g^{-1}$. Also, since $A$ is a finite algebra over
$R$, $A$ is a semilocal ring. Since $n$ is invertible in $R$, by
Vaserstein's result \cite{Vas}, the Dieudonn\`e determinant
extends to the setting of $M_{n^2}(A)$. Taking the determinant
from $f(a)$ and $g i(a) g^{-1}$, it follows that
$1=N_{A/R}(a)=a^{n^2}c_a$ where $c_a \in A'$. This shows that
$\SK(A)$ is $n^{2}$-torsion.

{\bf (iv).} {\it Platonov's Congruence Theorem: $\SL[](1,A) \cap (1+IA)
\subseteq A'$.}

Let $a \in \SL[](1,A)\cap (1+IA)$. By part (i), there is $b \in 1+IA$
such that $b^{n^2}=a$. Then $\Nrd_{A}(a)=\Nrd_A(b)^{n^2}=1$. By
part (ii), $\Nrd_A(b) \in 1+I$ and since $1+I$ is uniquely
$n$-divisible, $\Nrd_A(b)=1$, so $b \in \SL[](1,A)$. By part (iii),
$b^{n^2} \in A'$, so $a \in A'$. Thus $\ker \phi \subseteq A'$
where $\phi:\SL[](1,A) \rightarrow \SL[](1,\overline A)$.

It is easy to see that $\phi$ is surjective. In fact, if $\overline a
\in \SL[](1,\overline A)$ then $1=\Nrd_{\overline A}(\overline
a)=\overline{\Nrd_{A}(a)}$ thus, $\Nrd_{A}(a) \in 1+I$. By part (i),
there is $r \in 1+I$ such that $\Nrd_A(a r^{-1})=1$ and
$\overline{a r^{-1}}=\overline a$. Thus $\phi$ is an epimorphism.
Consider the induced map $\overline \phi: \SL[](1,A) \rightarrow
\SL[](1,\overline A)/ \overline A'$. Since $ I \subseteq J(R)$, and by part (iii),
$\ker \phi \subseteq A'$ it follows that $\ker \overline \phi = A'$ and
thus $\overline \phi: \SK(A) \cong \SK(\overline A)$.
\end{proof}

Let $R$ be a semilocal ring and $(R, J(R))$ a Hensel pair. Let
$A$ be an Azumaya algebra over $R$ of constant rank $n$ and $n$
invertible in $R$. Then by Theorem~\ref{main}, $\SK(A) \cong
\SK(\overline A)$ where $\overline A= A/J(R)A$. But $J(A)=J(R)A$, so $\overline A
= M_{k_1}(D_1)\times \cdots M_{k_r}(D_r)$ where $D_i$ are division
algebras. Thus $\SK(A)\cong \SK(\overline A)= \SK(D_1) \cdots \times
\SK(D_r)$.

Using a result of Goldman~\cite{goldman}, one can remove the
condition of Azumaya algebra having a constant rank from the
Theorem.

\begin{corollary}\label{nonconstant}
Let $A$ be an Azumaya algebra over a Hensel pair $(R,I)$ where $R$
is semilocal and the least common multiple of local ranks of $A$
over $R$ is invertible in $R$.  Then $\SK(A) \cong \SK(\overline A)$
where $\overline A=A/IA$.
\end{corollary}
\begin{proof}
One can decompose $R$ uniquely as $R_1\oplus \cdots \oplus R_t$
such that $A_i=R_i\otimes_R A$ have constant ranks over $R_i$
which coincide with local ranks of $A$ over $R$ (see
\cite{goldman}, \S2 and Theorem 3.1). Since $(R_i,IR_i)$ are
Hensel pairs, the result follows by using Theorem~\ref{main}.
\end{proof}

\begin{remarks}
Let $D$ be a tame unramified division algebra over a Henselian
field $F$, i.e., the value group of $D$ coincides with value group of $F$ and 
$\chr(\overline F)$ does not divide the index of $D$ (see \cite{wadsval} for a nice survey on valued division algebras). Let $V_D$ be the valuation ring of $D$ and $U_D=V_D^*$. Jacob and Wadsworth observed that
$V_D$ is an Azumaya algebra over its center $V_F$ (Theorem 3.2 in
\cite{wadsval} and Example 2.4 in \cite{jacwad}). Since
$D^*=F^*U_D$ and $V_D \otimes_{V_F}F \simeq D$, it can be seen
that $\SK(D)=\SK(V_D)$. On the other hand our main
Theorem states that $\SK(V_D)\simeq \SK(\overline D)$.
Comparing these, we conclude the stability of $\SK$ under
reduction, namely $\SK(D) \simeq \SK(\overline D)$ (compare this with
the original proof, Corollary 3.13 in \cite{plat}). 

Now consider the group $\CK(A)=A^*/R^*A'$ for the Azumaya algebra $A$ over the Hensel pair $(R,I)$. A proof similar to Theorem 3.10 in \cite{azumayask1}, shows that
$\CK(A) \cong \CK(\overline A)$. Thus in the case of tame unramified division algebra $D$, one can observe that $\CK(D)\cong \CK(\overline D)$.

For an Azumaya algebra $A$ over a semilocal ring $R$, by (\ref{longexactse}) one has
$$R^*/\Nrd_A(A^*) \cong H_{\text{\`et}}^1(R,\SL[](1,A)).$$ If
$(R,I)$ is also a Hensel pair, then by the Grothendieck-Strano result, 
$$R^*/\Nrd_A(A^*)\cong  H_{\text{\`et}}^1(R,\SL[](1,A)) \cong H_{\text{\`et}}^1(\overline
R,\SL[](1, \overline A))\cong \ {\overline R}^*/\Nrd_{\overline A}(\overline A^*).$$ However specializing to a tame unramified division algebra $D$, the stability does not follow in this case. In fact 
for a tame and unramified division algebra
$D$ over a Henselian field $F$ with the valued group $\Gamma_F$
and index $n$ one has the following exact sequence (see
\cite{commalg}, Theorem 1):
$$1 \longrightarrow H^1(\overline F, \SL[](1,\overline D))
\longrightarrow H^1(F,\SL[](1,D)) \longrightarrow \Gamma_F/
n\Gamma_F \longrightarrow 1.$$

\end{remarks}

\forget

\section{$\mathcal D$-functors}

Let $\text{Az}(R)$ be the category of Azumaya algebras over the
commutative ring $R$ and $C(R)$ category of all commutative
$R$-algebras. Consider a functor $\mathcal F:\text{Az}(R)
\rightarrow \mathcal Func(C(R),\text{Ab})$. Thus for any Azumaya
$R$-algebra $A$, $\mathcal F_A:C(R) \rightarrow \text{Ab}$ is a
functor from $C(R)$ to the category of abelian groups. Consider
the following property for $\mathcal F$:

(1) $\mathcal F_R$ is a trivial functor, i.e, $\mathcal F_R(S)=1$
for any $R$-algebra $S$.

And the following properties at the base point $R$,

(2) For any faithfully projective $R$-module $P$, there is a
homomorphism $d: \mathcal F_{A \otimes \End_R(P)}(R) \rightarrow
\mathcal F_A(R)$ such that the composition $\mathcal F_A(R)
\rightarrow \mathcal F_{A \otimes \End_R(P)}(R) \rightarrow
\mathcal F_A(R)$ is $\eta_{[P:R]}$ if $P$ has a constant rank
$[P:R]$. Here $\eta_k(x)=x^k$.

(3) If $P$ has a constant rank over $R$ then $\ker(d)$ is
$[P:R]$-torsion.

We call the functor $\mathcal F$ with the above properties a
$\mathcal D$-functor.

\begin{remark}
1. When evaluating $\mathcal F$ at the base point, i.e.,
$\mathcal F_A(R)$, one can drop $R$ and simply write $\mathcal
F(A)$. With this simplification, when $P$ is a free module of rank
$n$, the condition (2) says the composition of $\mathcal F(A)
\rightarrow \mathcal F(M_n(A)) \rightarrow F(A)$ is $\eta_n$.

2. The goal is to show that $\mathcal F_A(R)$ is a torsion abelian
group. However one would like to prove the same statement for any
$\mathcal F_A(S)$ where $S$ is an $R$-algebra. This could be done
if one further assumes

(4) For any $S \in C(R)$, $\mathcal F_A(S)=\mathcal F_{A \otimes
S}(S)$.

Here $\mathcal F_{A \otimes S}$ is a functor from $\text{Az}(S)
\rightarrow \mathcal Func(C(S),\text{Ab})$. There is no need to
introduce an extra notation, say, $\mathcal F_A^R$ and $\mathcal
F_{A\otimes S}^S$ here. The Condition (4) simply states that
evaluating $\mathcal F_A$ at $S$ is the same as evaluation the
functor $\mathcal F_{A \otimes S}$ at the base point.

\end{remark}

\begin{lemma} Let $\mathcal F$ be a $\mathcal D$-functor. Then
$\mathcal F_A(R)$ is $n$-torsion.
\end{lemma}
\begin{proof}
\end{proof}

\begin{theorem}
Let $\mathcal F$ be a $\mathcal D$-functor. If $A$ and $B$ are
Azumaya algebras with relatively prime ranks, then $\mathcal F_{A
\otimes B}= \mathcal F_A \times \mathcal F_B$.
\end{theorem}

{\bf Aim:} If $A$ is an Azumaya algebra over $R$ of constant rank $n$, then $$K_i(A)\tensor_{\Bbb Z} \Bbb Z[1/n] \cong K_i(R)\tensor_{\Bbb Z} \Bbb Z[1/n].$$  

So far this can be done if $A$ is Azumaya over semilocal ring, as in this case constant rank  projective modules are free, thus by previous paper.

\forgotten

{\it Acknowledgement.} I would like to thank IHES, where part of
this work has been done in Summer 2006 and the support of EPSRC first grant
scheme EP/D03695X/1.

\end{document}